\documentclass[12pt]{amsart}
 
\usepackage{etex}
\usepackage[utf8]{inputenc}
\usepackage[T1]{fontenc}
\usepackage{geometry}
\usepackage{lmodern}
\usepackage{amsmath}
\usepackage{amsthm}
\usepackage{amssymb}
\usepackage{mathrsfs}
\usepackage{graphicx}
\usepackage{pst-all}
\usepackage{tikz}
\usepackage{marvosym}
\usetikzlibrary{shapes}

\title{Three-way symbolic tree-maps and  ultrametrics}

\author{K.T.Huber, V.Moulton, G.E.Scholz}
\address{Huber, Moulton, Scholz: School of Computing Sciences, University of East Anglia, UK. 
	e-mail: k.huber@uea.ac.uk, v.moulton@uea.ac.uk, g.scholz@uea.ac.uk}

\date
	{\today}

\newtheorem{theorem}{Theorem}[section]
\newtheorem{lemma}[theorem]{Lemma} 
\newtheorem{corollary}[theorem]{Corollary}
\newtheorem{proposition}[theorem]{Proposition}

\begin{document}

\maketitle

\begin{abstract}
	    Three-way dissimilarities are a generalization of (two-way) dissimilarities which 
		can be used to indicate the lack of homogeneity or resemblance between any three objects.
		Such maps have applications in cluster analysis, and have been used in areas such as
		psychology and phylogenetics, where three-way data tables can arise. 
		Special examples of such dissimilarities are three-way tree-metrics and ultrametrics, 
		which arise from leaf-labelled trees with edges labelled by positive real numbers. Here we consider
		three-way maps which arise from leaf-labelled trees where instead
		the interior {\em vertices} are labelled by an {\em arbitrary} set of values.
		For unrooted trees we call such maps {\em three-way  symbolic tree-maps}; for rooted trees
		we call them {\em three-way symbolic ultrametrics} since they can be 
		considered as a generalization of the (two-way) 
		symbolic ultrametrics of B\"{o}cker and Dress. 
		We show that, as with two- and three-way 
		tree-metrics and ultrametrics, three-way symbolic tree-maps 
		and ultrametrics can be characterized via certain $k$-point conditions. 
		In the unrooted case, our characterization is mathematically
		equivalent to one presented by Gurvich
		for a certain class of edge-labelled hypergraphs. We also 
		show that it can be decided whether or not an arbitrary three-way symbolic map 
		is a tree-map or a symbolic ultrametric using a triplet-based approach that relies on the so-called
		BUILD algorithm for deciding when a set of 3-leaved trees or {\em triplets} can be displayed by a single tree. We envisage that our results will be useful in 
		developing new approaches and algorithms for understanding 3-way data, especially within
		the area of phylogenetics.
\end{abstract}


\section{Introduction}

Three-way dissimilarities are a generalization of (two-way) dissimilarities which 
can be used to indicate the lack of homogeneity or resemblance between any three objects 
in a given set \cite{JC}. They have applications in areas such as
psychology \cite{HB97} and phylogenetics \cite{L06}, where they have been  used
to cluster data presented in the form of three-way data tables. 
Various special classes of three-way dissimilarities have 
been introduced (see e.g. \cite{chepio,H72,HB97,JC}). These 
include three-way dissimilarities that  arise from leaf-labelled trees, 
where the edges are weighted by positive real numbers.
These so-called {\em three-way tree-metrics} and 
{\em three-way  ultrametrics}, which arise from unrooted and rooted
trees, respectively, generalize their much studied two-way 
counterparts (cf. \cite{chepio} for an overview).

Intriguingly, in \cite{BD98} B\"ocker and Dress showed
that the concept of ultrametricity for dissimilarities can be naturally
extended to include two-way symmetric maps whose range is an arbitrary set of symbols.
In particular, they introduced the concept of a {\em symbolic ultrametric} (a two-way 
map arising from a rooted, vertex-labelled tree via the least common ancestor map), 
and characterized them in terms of a 3- and a 4-point condition (see Section~\ref{sec:prelim}
for full details), a result which had in fact been discovered
independently in another guise by  V.~Gurvich \cite{G84} (see Section~\ref{sec:prelim}
for details). These conditions generalise the well-known 3-point 
condition for ultrametricity (cf. e.g. \cite[Chapter 7.2]{SS03}). Symbolic ultrametrics have 
been found to have interesting connections with cograph theory \cite{H13}, game theory \cite{G84,G09}, as well as
applications within phylogenetics \cite{Stadler,LM15}. Therefore, it is of interest to understand 
how the theory of symbolic ultrametrics can be extended to three-way maps, as these may lead to useful new applications in these areas (e.g. see the last section for a potential application in phylogenetics).

In this paper, we shall address this question. 
Let $X$ be a set (of taxa) of size at least 3 and let
$M$ be a set (of {\em symbols}) of size at least 2.  
A \emph{(three-way) symbolic map} is a map $\delta : {X \choose 3} \to M$. 
For example, consider the unrooted tree in Figure~\ref{illustrate}(a)
with leaf-set $X=\{1,2,3,4,5\}$ whose interior vertices are
labelled by elements from the set $M=\{A,B\}$. This tree 
gives rise naturally to a three-way symbolic map from $X$ to 
the set $M$; to each triple of leaves
we assign the element of $M$ which labels the vertex lying 
on all shortest paths between any two of these three leaves
(e.g. the triple $\{1,3,5\}$ is assigned the symbol $A$). We call
symbolic maps that arise in this way {\em three-way symbolic tree-maps}.

\begin{figure}[h]
	\begin{center}
		\includegraphics[scale=0.7]{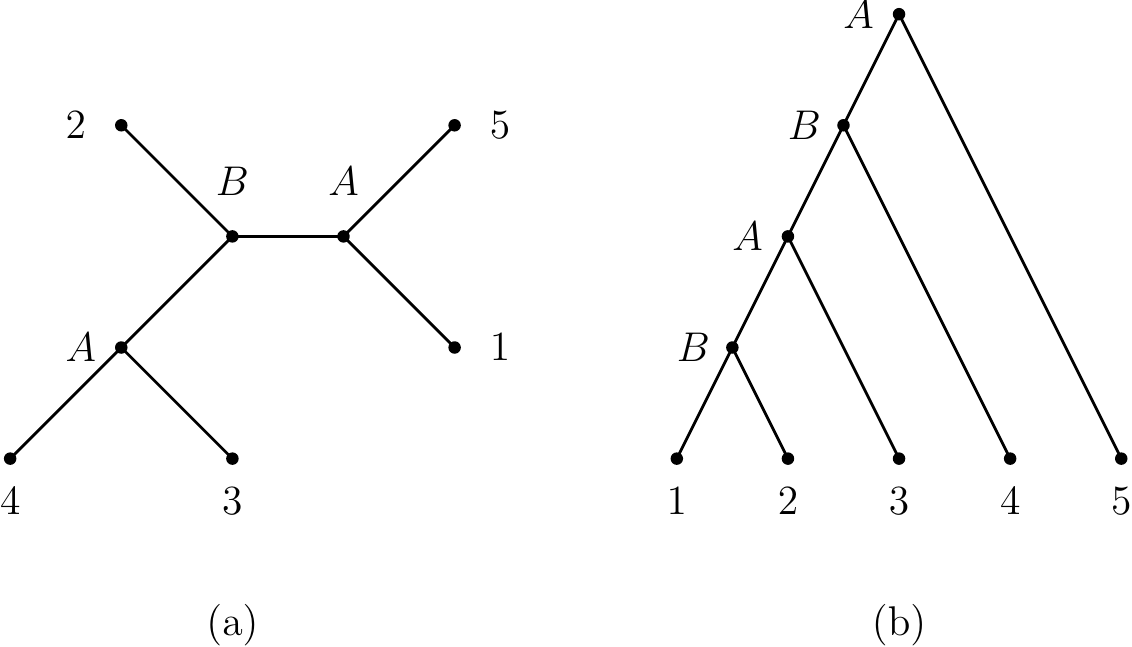}
		\caption{Two trees which give rise to (a) a three-way symbolic tree-map and (b) a three-way 
			symbolic ultrametric.}
		\label{illustrate}
	\end{center}
\end{figure}

In Section~\ref{sec:tree} we show that a three-way symbolic tree-map 
uniquely determines its underlying 
labelled tree (Proposition~\ref{treeunique}), and 
also give a 4- and 5-point characterization for such maps 
(see Theorem~\ref{th3stm}).
This result is mathematically equivalent to \cite[Theorem 5]{G84},  but 
for completeness we provide its proof.  Our characterization for three-way symbolic tree-maps is 
analogous to the well-known 4-point condition
for tree-metrics (cf. \cite[Chapter 7.1]{SS03}), and also generalizes the conditions 
presented in  \cite[Theorem 7]{HHMS12} for determining 
when a three-way dissimilarity arises from a tree. To prove Theorem~\ref{th3stm}
we introduce a symbolic variant of the Farris transform \cite[p.149]{SS03}, which allows
us to apply the main result from \cite{BD98}. We conclude the section with a 
description of how our result is related to the ones presented in \cite{G84}.

In Section~\ref{sec:ultra}, we turn our attention to obtaining 
three-way symbolic maps from rooted trees. Consider the 
rooted tree in Figure~\ref{illustrate}(b). A symbolic ultrametric
can be associated to this tree by defining the value for each pair 
of leaves to be the symbol labelling the least common ancestor vertex of
these two leaves. Therefore, a natural way to define a  
three-way symbolic ultrametric
could be to take the value of each triple of leaves to 
be the {\em set} consisting of the symbols
labelling the least common ancestor of all 
pairs of leaves in the triple (for example, we would assign the set 
$\{A,B\}$ to the triple $1,2,5$). However, this does not suffice to capture 
the tree (see Section~\ref{sec:ultra}). 

Even so, as we shall see, if we
consider the values of the triples to be {\em multisets}
instead of sets (for example, we would
assign the multiset $\{A,A,B\}$ to the triple $1,2,5$ in Figure~\ref{illustrate}(b)), then we can in fact recover 
the underlying labelled tree in case $|X| \ge 5$ (Theorem~\ref{5enough}).
We call maps obtained in this way {\em three-way symbolic ultrametrics}.
In Section~\ref{sec:five}, we give 3-, 4- and 5-point conditions
which ensure that a three-way symbolic map that maps into the set of
size 3 multisets of a set of symbols is a symbolic ultrametric. This 
is somewhat surprising since for three-way dissimilarities, a 6-point 
condition is required to ensure that they can be represented by  a rooted
tree in an analogous way (cf. \cite[Theorem 7]{HHMS12}).

We conclude the paper by considering an alternative approach
for deciding whether or not a three-way symbolic map is a tree-map or symbolic ultrametric.
This approach is based on the BUILD algorithm \cite{A81}, which 
can be used to decide when a set of {\em triplets} 
(i.\,e.\, resolved rooted leaf-labelled trees each with three leaves)
is displayed by some supertree or not. Applying this algorithm to  
three-way symbolic maps has the advantage
that only sets of size three (as opposed to sets of size up to five) 
need to be considered so as to 
determine if a three-way symbolic map is a tree-map or a symbolic ultrametric. This 
could potentially lead to practical algorithms for performing this task. 
In Section~\ref{sec:discuss}, we present some future directions.

\section{Preliminaries}\label{sec:prelim}

For a set $\{x_1, \ldots, x_k\}$, $k\geq 1$, in the powerset $\mathcal{P}(X)$
of $X$ and a map $\delta : \mathcal P(X) \to M$, 
we will write $\delta(x_1, \ldots, x_k)$ instead of $\delta(\{x_1, \ldots, x_k\})$. 

A \emph{symbolic ultrametric} \cite{BD98} is a 2-way 
symbolic map $D : {X \choose 2} \to M$ satisfying:
\begin{enumerate}
\item[(U1)] For all three distinct elements $x,y,z \in X$, at least two of the 
three values $D(x,y)$, $D(y,z)$ and $D(x,z)$ are the same.
\item[(U2)] There exists no four pairwise distinct elements $x,y,z,u \in X$ such 
that $D(x,y)=D(y,z)=D(z,u) \neq D(z,x)=D(x,u)=D(u,y)$.
\end{enumerate}

Suppose that $T$ is a tree. Then we denote by $L(T)$ the
set of leaves of $T$ and by $V^o(T):=V(T)-L(T)$ 
the set of {\em internal vertices} of $T$. If $T$ is 
rooted then we denote by $\rho_T$ the root of $T$.
Moreover, for any two distinct leaves $x$ and $y$ in $T$, 
we define the {\em least common ancestor} $\mathrm{lca}_T(x,y)$ 
of $x$ and $y$ in $T$ to be the last vertex in $T$ that lies on both of 
the paths which start at  $\rho_T$ and end 
in $x$ and in $y$. If $\mathrm{lca}_T(x,y)$ is adjacent with both
$x$ and $y$ then we call the set $\{x,y\}$ a {\em cherry} of $T$.

We also say that vertex $v$ in $T$ 
lies {\em below} a vertex $w\not=v$ in $T$ 
if $w$ lies on the path from the root of $T$ to $v$.

A \emph{(rooted/unrooted) phylogenetic tree} $T$ 
on $X$ is a (rooted/unrooted) tree with leaf-set $X$
that does not contain vertices of degree two in case
$T$ is unrooted and no vertex with indegree and outdegree
one in case $T$ is rooted. Note that we will only use
the terms rooted or unrooted in case it is not 
clear  from the context which type of tree we are considering. 
Two phylogenetic trees $T$ and $T'$ on $X$  
are {\em isomorphic} if there exists a bijection
$V(T)\to V(T')$ that induces a graph isomorphism between
$T$ and $T'$ that is the identity on $X$  (i.e. the map which takes every element in $X$ to itself). 
In case $T$ is a rooted phylogenetic tree on $X$ and $Y$ 
is a subset of $X$ with size at least two, we
let $T_Y$ denote the phylogenetic tree spanned by $Y$ 
(obtained by suppressing vertices with indegree and
outdegree one), and say that $T_Y$ is {\em induced by $Y$}.
 
A \emph{labelled (rooted/unrooted) tree} $\mathcal T$ on $X$ is a pair $(T,t)$, where $T$ 
is a (rooted/unrooted) phylogenetic tree on $X$, and $t$ is a {\em labelling map on $M$},
that is, a map from 
the internal vertices of $T$ to a set $M$ of symbols. 
If $t(u) \neq t(v)$ for every $u \neq v$ contained in the
same edge of $T$, we say that 
$\mathcal T$ is \emph{discriminating}. 
A labelled rooted tree $\mathcal T=(T,t)$ on $X$ is a \emph{representation} 
of a (two-way) symbolic map $D : {X \choose 2} \to M$ (or $\mathcal T$ 
\emph{represents}  $D$) if for all distinct $x,y \in X$, we 
have $D(x,y)=t(\mathrm{lca}_T(x,y))$.
 
\begin{theorem}[B\"ocker and Dress, 1998]\label{thBD} 
	Let $D : {X \choose 2} \to M$ be a symbolic map. There exists a 
	discriminating labelled rooted tree $\mathcal T$ that represents $D$ if and only if $D$ 
	is a symbolic ultrametric. If this holds, then such a tree is necessarily unique.
\end{theorem}

Interestingly, Theorem~\ref{thBD} appeared in a different guise 
in \cite{G84} (see also \cite{G09} for more details) in the context of game theory. 
Within this context, the leaves of the tree $T$ are seen as 
end of game situations, the label set $M$ corresponds to a set of players, and a 
directed path from the root of $T$ to a 
leaf is a sequence of plays. As we will come back to this correspondence in 
Section~\ref{sec:tree}, we now review some relevant terminology and
results presented in \cite{G84}.

Suppose that $H$ is an \emph{edge-labelled graph} on $X$, 
that is a graph with vertex set $V(H)=X$ and edge set $E(H)={X \choose 2}$, 
equipped with a map $D: E(H) \to M$ for a given, nonempty set $M$. 
Then $D$ is a (two-way) symbolic map on $X$, and any symbolic 
map $D: {X \choose 2} \to M$ 
can trivially be seen as an edge-labelled graph $(H,D)$ on $X$.

An edge-labelled graph $(H, D)$ on $X$ is said to be \emph{linked} 
if for all $m \in M$, 
the graph $H_m$ obtained from $H$ by removing all edges $e \in E(H)$ 
for which $D(e)=m$ holds is 
connected. For example, both graphs $\Delta$ and $\Pi$ depicted in Figure~\ref{vpic} are linked. 
If $(H, D)$ does not contain any linked subgraph, it is said to be \emph{separated}.

\begin{figure}[h]
\begin{center}
\includegraphics[scale=0.8]{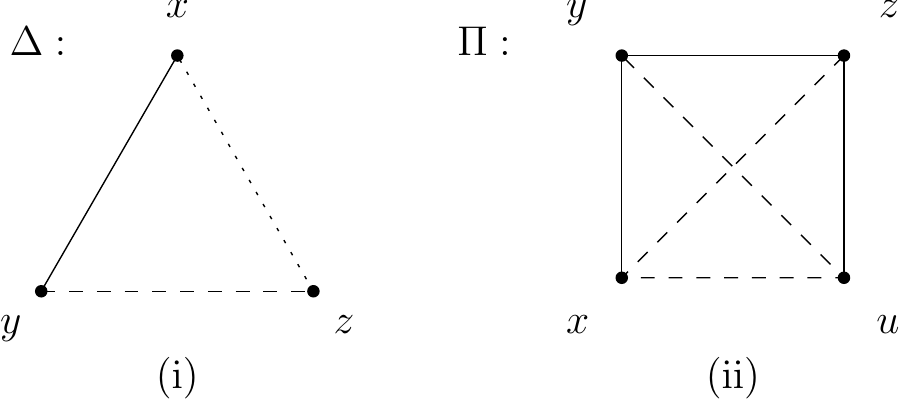}
\caption{(i) An edge-colored graph $\Delta$ on $X=\{x,y,z\}$
adapted from \cite[Figure 1]{G84}. (ii) An edge-colored 
	graph $\Pi$ on $X=\{x,y,z,u\}$ adapted from the same figure. 
Colors are represented in terms of different edge styles (plain, 
dashed and dotted). }
\label{vpic}
\end{center}
\end{figure}

The following result from \cite{G84} links the property for an edge labelled graph to 
be separated with the representability of the symbolic map it induces. 
For this, it relies on the equivalence between the following three 
statements for a 
symbolic map $D: {X \choose 2} \to M$ (see \cite[Theorem~2]{G84} for the 
equivalence between (ii) and (iii), and \cite[Theorem~4]{G84} for the equivalence 
between (i) and (ii), where a discriminating labelled rooted tree is 
called a \emph{positional structure}, or \emph{PS} for short):
\begin{itemize}
\item[(i)] There exists a (unique) discriminating labelled rooted tree $\mathcal T$ that represents $D$.
\item[(ii)] The edge-labelled graph $(H,D)$ is separated.
\item[(iii)] The edge-labelled graph $(H,D)$ does not contain any subgraph isomorphic to $\Delta$ or $\Pi$ (depicted in Figure~\ref{vpic}).
\end{itemize} 

As mentioned above, this result, and in particular the equivalence 
between conditions (i) and (iii), provides a direct equivalent to 
Theorem~\ref{thBD}. Indeed, as is easy to see a symbolic map $D: X^2 \to M$ 
satisfies (U1) (\emph{resp.} (U2)) if and only if the 
edge-labelled graph $(H,D)$ 
does not contain a subgraph isomorphic to $\Delta$ (\emph{resp.} $\Pi$), 
implying that condition (iii) and the property of being a symbolic ultrametric are equivalent.

\section{Three-way symbolic tree-maps}\label{sec:tree}

We begin by considering three-way symbolic maps that arise from labelled unrooted trees.
Such a tree $\mathcal T=(T,t)$ clearly gives rise to a three-way 
symbolic map $\delta_{\mathcal T} : {X \choose 3} \to M$
by putting, for all $x,y,z \in X$, $\delta_{\mathcal T}(x,y,z)=t(med_T(x,y,z))$, where
$med_T(x,y,z)$ denotes the {\em median vertex} 
of $x,y,z$ in $T$ (that is, the
unique vertex lying on the paths from $x$ to $y$, from $x$ to $z$ and from $y$ to $z$, respectively).

If for a three-way symbolic map $\delta: {X \choose 3} \to M$, there exists a labelled unrooted 
tree $\mathcal T$ such that $\delta=\delta_{\mathcal T}$, we say that $\delta$ 
is a {\em three-way symbolic tree-map (on $X$)}, 
and that $\mathcal T$ is a \emph{representation} of $\delta$
(or $\mathcal T$ \emph{represents} $\delta$). 
We now characterize such maps.
To do this, we define a \emph{symbolic Farris transform}, 
the definition of which is adapted from the well-known
Farris transform  \cite[p.149]{SS03} as follows. 

Suppose $\mathcal T=(T,t)$ is a labelled unrooted tree
 on $X$ where $|X|\geq 4$. Put
$\delta=\delta_{\mathcal T}$. Pick a leaf $r \in X$, and 
define a rooted phylogenetic tree $T_r$ on 
$X-\{r\}$ as follows: 
direct all edges of $T$ away from $r$, and remove $r$ 
and its outgoing edge. This induces a  bijection $\psi_r$ from 
the set of internal vertices of $T$ to the set of internal vertices of $T_r$. Hence 
the map $t_r: V(T_r)\to M$ 
which takes any internal vertex $v$ 
 of $T_r$ to $M$ given
by  $t_r(v)=t(\psi_r^{-1}(v))$ is well-defined, and 
the pair $\mathcal T_r=(T_r,t_r)$ is a labelled rooted 
tree.

Now, suppose that $\delta$ is the three-way symbolic tree-map that is 
represented by $\mathcal T$, and that $D_r$ is
the symbolic ultrametric on $X$ that is 
represented by $\mathcal T_r$. 

\begin{lemma}\label{lmddr}
For all $x,y \in X-\{r\}$ with $|X|\geq 4$, we have $D_r(x,y)=\delta(x,y,r)$.
\end{lemma}

\begin{proof}[Proof]
It suffices to note that via the symbolic Farris transform, 
the median vertex of $x$, $y$ and $r$ in $T$ becomes 
the least common ancestor of $x$ and $y$ in $T_r$. Denoting the latter by $v$, we then have $D_r(x,y)=t_r(v)=t(\psi_r^{-1}(v))=t(med_T(x,y,r))=\delta(x,y,r)$.
\end{proof}

Motivated by this observation, for a three-way 
symbolic map $\delta : {X \choose 3} \to M$ and some $r \in X$ where $|X|\geq 4$, we 
define the map 
$$\delta_r: {X-\{r\} \choose 2} \to M; \,\, \delta_r(x,y)=\delta(x,y,r),
$$ 
for all $x,y  \in X$
distinct (which can be considered as a symbolic analogue
of the Farris transform as defined in \cite[p.149]{SS03}).
Using Lemma~\ref{lmddr}, we can now prove a 
uniqueness result.

\begin{proposition}\label{treeunique}
Let $\delta: {X \choose 3} \to M$ be a three-way symbolic tree-map where $|X|\geq 4$. There 
exists a unique discriminating labelled unrooted tree $\mathcal T$ that represents $\delta$.
\end{proposition}

\begin{proof}[Proof]
Let $r \in X$ and consider the map $\delta_r: {X-\{r\} \choose 2} \to M$. 
By Lemma~\ref{lmddr}, $\delta_r$ is a symbolic ultrametric, and thus, admits a unique discriminating 
representation $\mathcal T_r$. Moreover, this representation is obtained from 
a representation of $\delta$, using the symbolic Farris transform. This operation 
is clearly invertible, and preserves the property of being discriminating. 
Thus, the labelled unrooted tree $\mathcal T$ obtained from $\mathcal T_r$ by 
inverting the symbolic Farris transform is necessarily 
the only discriminating representation of $\delta$.
\end{proof}

We now characterize three-way symbolic tree maps. As we shall 
explain below, an equivalent characterization appears in \cite{G84} 
in the guise of Theorem~5 of that paper. For the sake of completeness, 
we present a proof within our framework. Subsequent to
this, we explain how the approach in \cite{G84} relates to ours.

\begin{theorem}\label{th3stm}
Suppose that $|X|\geq 4 $ and that 
$\delta:{X \choose 3} \to M$ is a three-way symbolic map.
Then $\delta$ is a three-way symbolic tree-map 
if and only if $\delta$ satisfies the following two conditions: 
\begin{enumerate}
\item[(M1)] For all $\{x,y,z,u\} \in {X \choose 4}$, either
$$
\delta(x,y,z)=\delta(x,y,u)=\delta(x,z,u)=\delta(y,z,u)
$$
or two of these four are equal  and so are the remaining two.\\
\item[(M2)] There does not exist $\{x,y,z,u,v\} \in {X \choose 5}$ such that
$$
 \delta(v,x,y)=\delta(v,y,z)=\delta(v,z,u) \neq \delta(v,z,x) =\delta(v,x,u)= \delta(v,u,y).
$$
\end{enumerate}
\end{theorem}

In order to prove Theorem~\ref{th3stm}, we start with a useful lemma.

\begin{lemma}\label{lmtmu}
Suppose	that $|X|\geq 4 $  and that
$\delta:{X \choose 3} \to M$ is a three-way symbolic map 
satisfying (M1) and (M2). Then for all $r \in X$, the map $\delta_r$ is a symbolic ultrametric.
\end{lemma}

\begin{proof}[Proof]
Let $r \in X$. We need to show that $\delta_r$ satisfies properties (U1) and (U2).

To see that $\delta_r$ satisfies (U1), consider three elements $x,y,z \in X- \{r\}$.
Since $\delta$ satisfies (M1) the set $\{\delta(r,x,y),\delta(r,x,z),\delta(r,y,z)\}$
contains at most two distinct elements. As this set 
is precisely the set $\{\delta_r(x,y),\delta_r(x,z),\delta_r(y,z)\}$, (U1) follows.

To see that (U2) holds, assume for contradiction that there 
exist four pairwise distinct
elements $x,y,z,u \in X-\{r\}$ such that 
$\delta_r(x,y)=\delta_r(y,z)=\delta_r(z,u) \neq \delta_r(z,x)=\delta_r(x,u)=\delta_r(u,y)$. 
This implies $\delta(r,x,y)=\delta(r,y,z)=\delta(r,z,u) \neq \delta(r,z,x) =\delta(r,x,u)= \delta(r,u,y)$, 
which is impossible in view of (M2).
\end{proof}

Note that the converse of the Lemma~\ref{lmtmu} is not true in general. 
Consider for example the sets $X=\{1, \ldots, n\}$, $n \geq 4$, $M=\{A,B\}$, 
and the map $\delta:{X \choose 3} \to M$ defined for $x,y,z \in X$ by putting $\delta(x,y,z)=A$ 
if $1 \in\{x,y,z\}$ and $\delta(x,y,z)=B$ otherwise. 
Clearly, $\delta$ does not satisfy (M1), as we 
have $\delta(1,2,3)=\delta(1,2,4)=\delta(1,3,4) \neq \delta(2,3,4)$. 
However, we have $\delta_1(x,y)=A$ for all $x,y \in X-\{1\}$, 
which is clearly a symbolic ultrametric. 
In fact, for any $2 \leq k \leq n$ we have $\delta_k(x,y)=A$ if $1 \in \{x,y\}$ and $\delta_k(x,y)=B$ 
otherwise and, so, $\delta_k$ is also a symbolic ultrametric on $X-\{k\}$.

Armed with Lemma~\ref{lmtmu}, we can now prove Theorem~\ref{th3stm}.

\begin{proof}[Proof]
Assume first that $\delta$ is a three-way symbolic tree-map, and denote by $\mathcal T=(T,t)$ its representation.
To see that $\delta$ satisfies (M1), consider four 
pairwise distinct elements $x,y,z,u \in X$. Two cases may occur. If $med_T(x,y,z)=med_T(x,y,u)=med_T(x,z,u)=med_T(y,z,u)$, it follows immediately that $\delta(x,y,z)=\delta(x,y,u)=\delta(x,z,u)=\delta(y,z,u)$. Otherwise, there exists 
two pairs, say $\{x,y\}$ and $\{z,u\}$, such that the path between $x$ and $y$ and 
the path between $z$ and $u$ are disjoint. In this case, we have 
$med_T(x,y,z)=med_T(x,y,u) \neq med_T(x,z,u)=med_T(y,z,u)$. 
If $t(med_T(x,y,z))=t(med_T(x,z,u))$, it follows that $\delta(x,y,z)=\delta(x,y,u)=\delta(x,z,u)=\delta(y,z,u)$. 
Otherwise, we have $\delta(x,y,z)=\delta(x,y,u) \neq \delta(x,z,u)=\delta(y,z,u)$. Thus, $\delta$ satisfies (M1).

If $|X|=4$ then it is straight forward to check that
the theorem holds. So assume that $|X|\geq 5$.
To see that $\delta$ satisfies (M2), assume for contradiction that there exist pairwise distinct 
$x,y,z,u,v \in X$ such that $\delta(v,x,y)=\delta(v,y,z)=\delta(v,z,u) \neq \delta(v,z,x) =\delta(v,x,u)= \delta(v,u,y)$. We can apply the
symbolic Farris transform to $\mathcal T$ and $v$, thus obtaining a labelled rooted tree $\mathcal T_v$.
By Lemma~\ref{lmddr}, $\mathcal T_v$ is a representation of $\delta_v$, 
implying that $\delta_v$ is a symbolic ultrametric. But, by definition, 
$\delta_v$ satisfies $\delta_v(x,y)=\delta_v(y,z)=\delta_v(z,u) \neq \delta_v(z,x) =\delta_v(x,u)= \delta_v(u,y)$, 
which contradicts (U2).

Conversely, assume that $\delta$ satisfies Properties (M1) and (M2), and let $r \in X$.
By Lemma~\ref{lmtmu}, the map $\delta_r$ is a symbolic ultrametric. Thus there exists a labelled rooted
tree $\mathcal T_r=(T_r,t_r)$ on $X-\{r\}$ representing $\delta_r$. Consider the 
labelled unrooted tree $\mathcal T=(T,t)$ on
$X$ defined as follows. First, add a new vertex $r$ to $T_r$ and 
the edge $\{\rho_{T_r},r\}$. Then consider all edges 
in the resulting tree to be undirected. Let $t:V^o(T)\to M$
denote the map given by $t(v)=t_r(v)$, for all $v\in V^o(T)$.
 We claim that for all $\{x,y,z\} \in {X \choose 3}$, we have $\delta(x,y,z) = t(med_T(x,y,z))$, 
that is, $\mathcal T$ is a representation of $\delta$.
To prove this it suffices to consider two cases. Suppose $\{x,y,z\} \in {X \choose 3}$.

\noindent{\em Case (a):} $\{x,y,z\} \subseteq X-\{r\}$. Without loss of generality, 
$\delta_r(x,z)=\delta_r(y,z)=t_r(u)$ and $\delta_r(x,y)=t_r(v)$, where $u$ and $v$ are vertices of $T_r$, and $v$ is below
or equal to $u$ in $T_r$. In this case  $t(med_T(x,y,z))$ 
equals $t_r(v)$.
By (M1) and since $\delta(x,z,r)=\delta(y,z,r)$, we have $\delta(x,y,z)=\delta(x,y,r)=t_r(v)=t(med_T(x,y,z))$.
Thus, $\mathcal T$ is a representation of $\delta$ in this
case.

\noindent{\em Case (b):} $r \in \{x,y,z\}$, say $r=z$. If we denote by $v$ the least common ancestor of
$x$ and $y$ in $T_r$, then $t(med_T(x,y,z))=t_r(v)$. 
Hence
$\delta(x,y,r)=\delta_r(x,y)= t_r(v)=t(med_T(x,y,r))$.
Thus, $\mathcal T$ is a representation of $\delta$ in this
case, too.
\end{proof}

We next elaborate on the relationship between Theorem~\ref{th3stm}
and Theorem~5 in \cite{G84}. 
As mentioned in Section~\ref{sec:prelim}, a symbolic two-way map 
$D: {X \choose 2} \to M$ can be seen as an edge-labelled graph $(H,D)$. 
Similarily, a symbolic three-way map $\delta: {X \choose 3} \to M$ can be 
seen as an edge-labelled 3-hypergraph $(\mathcal H,\delta)$, where by 
3-hypergraph, we mean that the edges of $\mathcal H$ are sets of three
 vertices (instead of two for graphs). Within this context, the vertex set of a 
3-hypergraph $\mathcal H$ associated to a 3-way map $\delta$ is $X$, 
as in the case of two-way maps, and the edge set of 
$\mathcal H$ is ${X \choose 3}$.

The author of \cite{G84} uses as a starting point for his characterization the equivalence, presented in Section~\ref{sec:prelim}, between symbolic 2-way maps that can be represented by a rooted tree and edge-labelled graphs that do not contain any subgraph isomorphic to the graphs $\Delta$ or $\Pi$ (see Figure~\ref{vpic}). The idea underlylng \cite[Theorem 5]{G84} is the following. From an edge-labelled 3-hypergraph $(\mathcal H, \delta)$ on $X$, we can pick an element $r \in X$ and consider the edge-labelled graph $(H,\delta_r)$ on $X-\{r\}$. It is then possible to highlight three edge-labelled 3-hypergraph $\delta_2, \delta_3, \delta_4$ with four vertices, that get transformed into edge-labelled graphs isomorphic to $\Delta$ via that operation, and one edge-labelled 3-hypergraph $\pi$ with five vertices, that gets transformed into an edge-labelled graph isomorphic to $\Pi$. The equivalence between the representability of $\delta$ by a labelled unrooted tree and the representability of $\delta_r$ by a labelled rooted tree for all $r \in X$ then leads to the conclusion that an edge-labelled 3-hypergraph $(\mathcal H, \delta)$ is representable if and only if it does not contain a sub(hyper)graph isomorphic to any of $\delta_2, \delta_3,\delta_4$ and $\pi$. 

As it turns out, $(\mathcal H, \delta)$ contains a sub(hyper)graph isomorphic to one of $\delta_2, \delta_3,\delta_4$ (\emph{resp.} to $\pi$) if and only if $\delta$ does not satisfy (M1) (\emph{resp.} (M2)). This implies 
that  Theorem~\ref{th3stm} and Theorem~5 in \cite{G84} are equivalent.

Finally, note that a similar result also appears in \cite{GLW17}. However, the arguments used by the authors of \cite{GLW17} do not rely on the projection of a three-way map to a two-way map and of an unrooted tree to a rooted tree, as is the case both here and in \cite{G84}.

\section{Three-way symbolic ultrametrics}\label{sec:ultra}

In the last section, we considered the problem of deciding when a three-way symbolic map
arises from a labelled unrooted tree. In this section, we start to
consider this problem for their rooted counterparts.  
In particular, after defining the concept of a three-way symbolic ultrametric,
we shall show that to determine whether or not 
a three-way symbolic map is a symbolic ultrametric, it suffices to consider 
its restriction to sets of size five.  

We begin by considering how to define
a three-way symbolic ultrametric.
If we consider 3 distinct leaves $x,y,z$ of a rooted phylogenetic tree $T$ on $X$, then
we can clearly identify two internal vertices of the tree given by the set 
$\{\mathrm{lca}_T(x,y), \mathrm{lca}_T(x,z), \mathrm{lca}_T(y,z)\}$ (in contrast to 
unrooted phylogenetic trees where we can identify only one, namely
the median of the 3 leaves).   A natural approach to
obtain a three-way symbolic map $\delta$ from a labelled
rooted  tree $\mathcal T=(T,t)$ might 
therefore be to 
take $\delta(x,y,z)$ to be the set $\{t(\mathrm{lca}_T(x,y)), t(\mathrm{lca}_T(x,z)), t(\mathrm{lca}_T(y,z))\}$,
for $x,y,z\in X$ distinct.
However, as can be seen in  Figure~\ref{figsid} such a map does not
necessarily uniquely capture $\mathcal T$.
For this reason, we shall consider instead  maps to multisets. 

\begin{figure}[h]
\begin{center}
\includegraphics[scale=0.7]{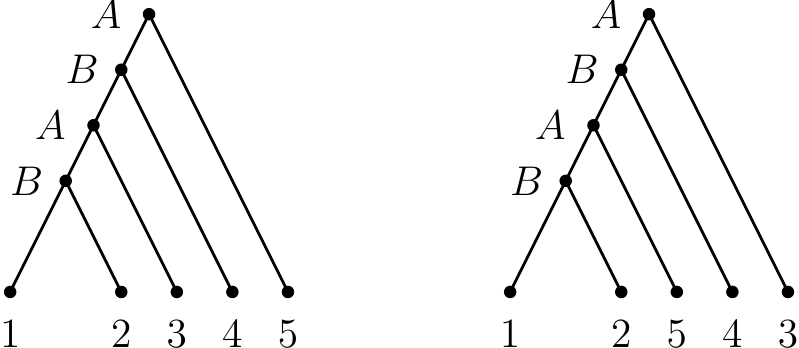}
\caption{Two labelled rooted trees on $X=\{1,2,3,4,5\}$ with labelling maps
	$t$ and $t'$ on $M=\{A,B\}$, respectively, for
	which the sets 
	$\{t(\mathrm{lca}_T(x,y)), t(\mathrm{lca}_T(x,z)), 
t(\mathrm{lca}_T(y,z))\}$ and
	$ \{t'(\mathrm{lca}_T(x,y)), t'(\mathrm{lca}_T(x,z)), t'(\mathrm{lca}_T(y,z))\}$ coincide,  for any three elements $x,y,z \in X$ 
	distinct.
\label{figsid}
}
\end{center}
\end{figure}

To formalize this, let
$\mathcal M=\mathcal M_M$
denote 
the set 
of multisets $\{a,b,c\}$ with $ a,b,c \in M$. As it will be useful later on, 
we shall also sometimes denote an element in $\mathcal M$ as  a sum. So, for example, for the 
element $\{a,a,b\} \in \mathcal M$ with $a,b\in M$, we sometimes also write $2a+b$. 

Now, given a labelled rooted 
tree $\mathcal T=(T,t)$ on $X$, we define the three-way symbolic map 
$\delta_{\mathcal T}: {X \choose 3} \to \mathcal M$ by putting 
$$
\delta_{\mathcal T}(x,y,z)=\{t(\mathrm{lca}_T(x,y)),t(\mathrm{lca}_T(x,z)),t(\mathrm{lca}_T(y,z))\}.
$$
for all distinct $x,y,z \in X$. If for a three-way symbolic map 
$\delta: {X \choose 3} \to \mathcal M$
there exists a labelled rooted 
tree $\mathcal T=(T,t)$ on $X$ such that $\delta = \delta_{\mathcal T}$,
then we call $\delta$ a \emph{three-way symbolic ultrametric (on $X$)}. Thus, intuitively, $\delta$ is a three-way symbolic ultrametric if it can be represented by labelling a rooted tree on $X$ in such a way that, for every 3-subset $\{x,y,z\}$ of $X$, $\delta(x,y,z)$ is the multiset consisting of the labels of the least common ancestors for all pairs of elements in $\{x,y,z\}$.
In addition, we say that $\mathcal T$ is a 
{\em representation} for $\delta$ (or that $\mathcal T$ {\em represents} $\delta$). 
We say that $\mathcal T$ is {\em discriminating} if $t(u)\not=t(v)$,
for every $u \neq v$ contained in the same edge in $T$.
Note that we can think of $\delta$ as a symbolic
analogue of a three-way perimeter map which arises from a weighted tree $T$
by taking, for any three leaves of $T$, the length of subtree
spanned by the those  leaves (see e.g.\cite{chepio}). 
Also, note that  by (U1), $\delta$
must satisfy the following property:

\begin{lemma}\label{lms2}
	Let $\delta : {X \choose 3} \to 
	\mathcal M$
	be a three-way symbolic 
	ultrametric. Then, for any three distinct elements $x,y,z \in X$, the number of 
	distinct elements in the multiset $\delta(x,y,z)$ is at most two.
\end{lemma}

We now turn our attention to showing that we can  determine whether or not 
a three-way symbolic map $\delta: {X \choose 3} \to \mathcal M$ 
is a symbolic ultrametric by restricting $\delta$ to subsets of $X$ with size five.  
To do this, we first need  to introduce 
some additional notation. For a subset $Y$ of $X$ of size four or more, 
let $\delta|_Y$ denote the restriction of 
$\delta$ to ${Y \choose 3}$,  that is, the map obtained by restricting the map $\delta$ to the subset $Y \choose 3$ of $X \choose 3$.. 
Note that  if $\delta$ is a three-way symbolic ultrametric, then 
$\delta|_{Y}$ is a three-way  symbolic ultrametric for all subsets $Y \subseteq X$ 
with $|Y| \ge 4$. Indeed, if $\mathcal T$ is a representation of $\delta$, then the 
subtree $\mathcal T_Y$ of $\mathcal T$ induced by $Y$ is a 
representation of $\delta|_Y$. Furthermore, we obtain a discriminating 
representation of $\delta|_Y$ by collapsing all edges of $\mathcal T_Y$  
both of whose end vertices have the same label.

\begin{figure}[h!]
\begin{center}
\includegraphics[scale=0.7]{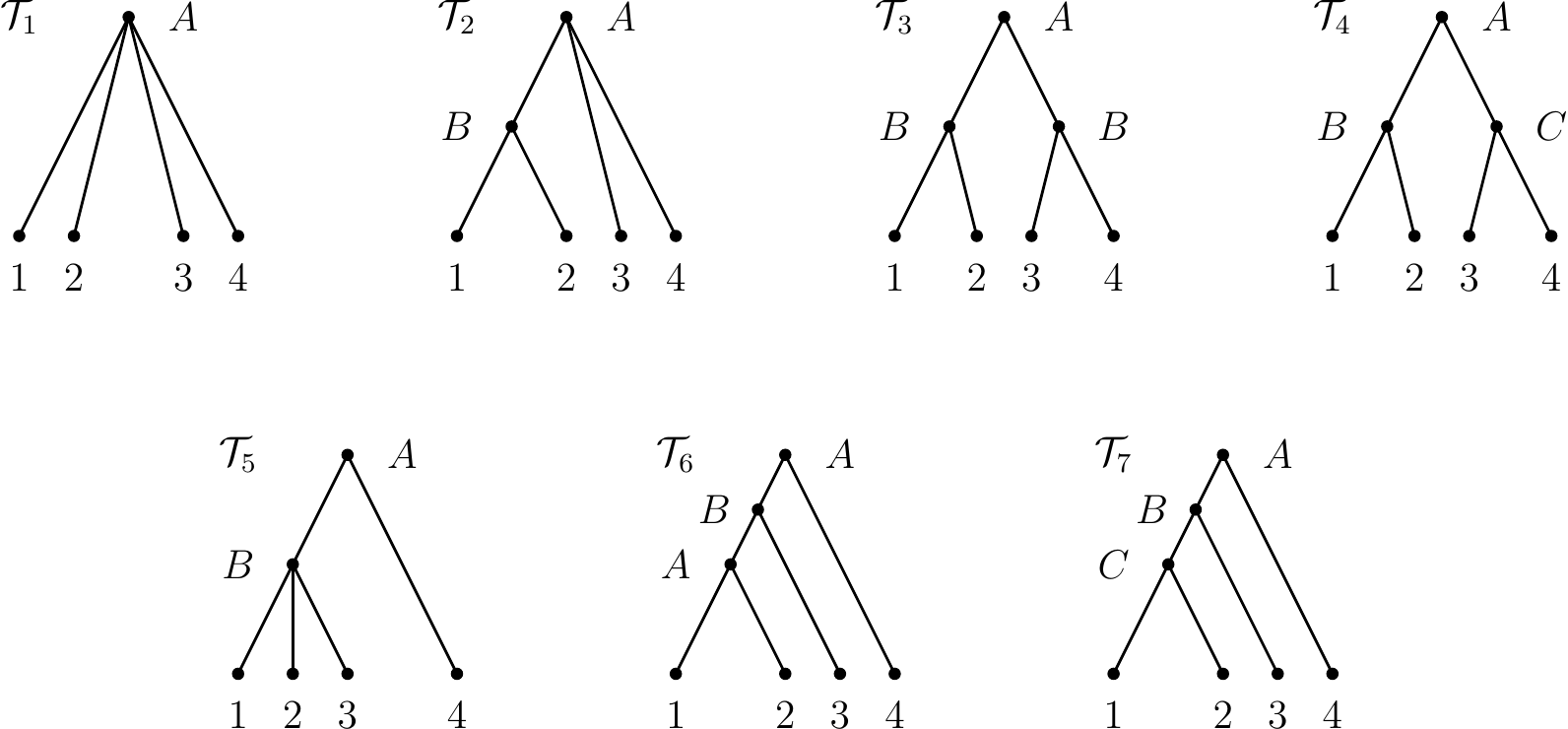}
\caption{All possible discriminating labelled rooted trees $\mathcal T_i$, $1\leq i\leq 7$,  
on $\{1,2,3,4\}$, up to a relabelling of the leaves.}
\label{fsub4}
\end{center}
\end{figure}

\begin{table}[h]
\begin{center}
\begin{tabular}{|c|c|c|c|c|c|c|c|}
\hline
$i$ & $1$ & $2$ & $3$ & $4$ & $5$ & $6$ & $7$ \\
\hline
\hline
$\hat{\delta}_i(1,2,3)$ & 3A & 2A+B & 2A+B & 2A+B & 3B & A+2B & 2B+C \\
\hline
$\hat{\delta}_i(1,2,4)$ & 3A & 2A+B & 2A+B & 2A+B & 2A+B & 3A & 2A+C \\
\hline
$\hat{\delta}_i(1,3,4)$ & 3A & 3A & 2A+B & 2A+C & 2A+B & 2A+B & 2A+B \\
\hline
$\hat{\delta}_i(2,3,4)$ & 3A & 3A & 2A+B & 2A+C & 2A+B & 2A+B & 2A+B \\
\hline
\end{tabular}
\caption{\label{tsub4}{For $1 \leq i \leq 7$ and
		$M=\{A,B,C\}$, the values of the map $\hat{\delta}_i$ represented 
		by the labelled rooted trees  $\mathcal T_i$  in Figure~\ref{fsub4}. 
		The trees $\mathcal T_i$ are given in terms of their index $i$ in the top row.
		}}
\end{center} 
\end{table}

We now  consider symbolic ultrametrics on a set of size four.
For  $M=\{A,B,C\}$, in Figure~\ref{fsub4}, we picture 
all possible discriminating  labelled rooted trees 
$\mathcal T_i$, $1\leq i\leq 7$,  on
$\{1,2,3,4\}$ and in Table~\ref{tsub4},
we list for all $1\leq i\leq 7$ the values of the map
$\hat{\delta_i}:{X \choose 3} \to \mathcal M$ that is 
represented by $\mathcal T_i$.
As we can see from this table, 
all of the maps $\hat{\delta}_i$ except for 
$\hat{\delta}_3$ capture  $\mathcal T_i$
(in the sense that $\mathcal T_i$ is the unique 
labelled rooted tree on $\{1,2,3,4\}$ that represents $\hat{\delta}_i$). 
Now, for $Y\subseteq X $ 
of size four, we say that $\delta|_Y$ 
is \emph{of type} $\hat{\delta}_i$, $i \in \{1, \ldots, 7\}$ if there 
exists a bijection between $Y$ and $\{1,2,3,4\}$ that
induces a bijection 
between the image of $\delta|_Y$ and the image 
of $\hat{\delta}_i$ such that $\delta|_Y$ and $\hat{\delta}_i$ coincide up to 
these bijections.  Since Table~\ref{tsub4} is exhaustive, we have:

\begin{proposition}\label{rep4}
Suppose that $|X|\geq 4 $, that $\delta : {X \choose 3} \to 
\mathcal M$
is a 
three-way symbolic map, and that $Y\subseteq X$ is 
a subset of size four. Then 
$\delta|_Y: {Y\choose 3}\to \mathcal M$ 
is a three-way symbolic ultrametric on $Y$ if and only 
if there exists some $i \in \{1, \ldots, 7\}$ such that $\delta|_Y$ is of type $\hat{\delta}_i$. 
Moreover, if $i \neq 3$, the representation of $\delta|_Y$ is unique.
\end{proposition}

\begin{table}[h]
\begin{center}
\begin{tabular}{|c|c||c|c|}
\hline
$\delta(1,2,3)$ & 2A+B & $\delta(1,4,5)$ & 2A+B \\
\hline
$\delta(1,2,4)$ & 2A+B & $\delta(2,3,4)$ & 3A \\
\hline
$\delta(1,2,5)$ & 3B & $\delta(2,3,5)$ & 2A+B \\
\hline
$\delta(1,3,4)$ & 3A & $\delta(2,4,5)$ & 2A+B \\
\hline
$\delta(1,3,5)$ & 2A+B & $\delta(3,4,5)$ & A+2B \\
\hline
\end{tabular}
\caption{\label{not4}{For $M=\{A,B\}$ 
		and $X=\{1,2,3,4,5\}$, a 
		three-way symbolic map 
		$\delta : {X \choose 3} \to \mathcal M$
		which is not a three-way symbolic ultrametric on $X$ 
		but whose restriction to any subset 
		$ Y \subset X$ of size four is a three-way 
		symbolic ultrametric on $Y$.}}
\end{center} 
\end{table}

We now  turn our attention to symbolic ultrametrics on
 sets of size five.
In the last result we have seen that a three-way symbolic ultrametric on a set of size 4
may have more than one representation by a labelled tree. 
However, as we shall now show this can not happen for sets of size five.

\begin{lemma}\label{lm5u}
If $Y$ is a set of size five and $\delta : {Y \choose 3} \to \mathcal M$
is a three-way symbolic ultrametric on $Y$, then $\delta$ has a unique discriminating representation. 
\end{lemma}

\begin{proof}[Proof]
Suppose that $\delta$ is a three-way 
symbolic ultrametric on $Y$,
and that $\mathcal T$ is a discriminating representation of $\delta$. 
Let $D=D_{\mathcal T}: {Y \choose 2} \to M$ 
be the symbolic ultrametric represented by $\mathcal T$. 
By Theorem~\ref{thBD},
it suffices to show that if $\delta$ is also represented by a labelled tree $\mathcal T'$,
then $D_{\mathcal T'} = D_{\mathcal T}$. 

Since $\delta$ is a three-way symbolic ultrametric on $Y$,
there exists a subset $Y_0$ of $Y$ with $|Y_0|=4$
such that $\delta|_{Y_0}$ is not of type $\hat{\delta}_3$. 
Thus, by Proposition~\ref{rep4}, $D_{\mathcal T'}|_{Y_0} = D_{\mathcal T}|_{Y_0}$. 
Hence, $D_{\mathcal T'}(x_0,x) = D_{\mathcal T}(x_0,x)$ for all
$x \in Y_0$ where $x_0$ is the unique element contained in $Y-Y_0$, since
the value of  $D_{\mathcal T'}(x_0,x)$ 
is given by  $\delta$ 
and $D_{\mathcal T}|_{Y_0}$ as follows. Let $Y_0=\{x,y,z,u\}$ and consider the multisets $\delta(x,y,x_0)-D_{\mathcal T}|_{Y_0}(x,y)$,
$\delta(x,z,x_0)-D_{\mathcal T}|_{Y_0}(x,z)$ and $\delta(x,u,x_0)-D_{\mathcal T}|_{Y_0}(x,u)$ where for a multiset $A$ with $k\geq 1$ copies of some
element $a$, we denote by $A-a$ the multiset obtained by
 removing one copy of $a$. 
If there exists a unique element $c \in M$ that belongs to
all three of these sets, we have $D_{\mathcal T'}(x_0,x)=c$. If two distinct elements of $M$ share this property, this implies $D_{\mathcal T'}(x_0,y)=D_{\mathcal T'}(x_0,z)=D_{\mathcal T'}(x_0,u) \neq D_{\mathcal T'}(x_0,x)$. We then have $D_{\mathcal T'}(x_0,y)=m(\delta(y,z,x_0))$, and $D_{\mathcal T'}(x_0,x)$ is the single element of $\delta(x,y,x_0)-\{D_{\mathcal T}|_{Y_0}(x,y),D_{\mathcal T'}(x_0,y)\}$.
\end{proof}

Note that, as the example in Table~\ref{not4} shows, it is not true in general 
that a three-way symbolic map $\delta$ that restricts to a three-way symbolic 
ultrametric on all subsets $Y$ of $X$ 
of size four is a three-way symbolic ultrametric on $X$.  However, 
as mentioned above, using the previous lemma we now show 
that considering sets of
size five is enough to ensure that this is the case.

\begin{theorem}\label{5enough}
Suppose that $|X|\geq 5 $ and that  
$\delta : {X \choose 3} \to \mathcal M$
is a three-way symbolic map. Then, $\delta$ is a 
three-way symbolic ultrametric if and only if $\delta|_Y$ is a  
three-way symbolic ultrametric for all $Y \subseteq X$ of 
size five.
\end{theorem}

\begin{proof}[Proof]
The fact that a three-way symbolic ultrametric on $X$ restricts to such an ultrametric
on all subsets of $X$ of size five is clear.

Conversely, assume that $\delta|_Y$ is a three-way symbolic 
ultrametric for all $Y \subseteq X$ of size five. 
For such a set $Y$, we denote by 
$\mathcal T_Y=(T_Y,t_Y)$ the unique 
(by Lemma~\ref{lm5u}) 
discriminating labelled tree that 
represents $\delta|_Y$, and by $D_Y$ the symbolic 
ultrametric that is represented by $\mathcal T_Y$.

Clearly, if there exists a map $D : {X \choose 2} \to M$ such that, 
for all subsets $Y \subseteq X$ of size five, the restriction of $D$ to ${Y \choose 2}$ 
coincides with $D_Y$, then $D$ satisfies $\delta(x,y,z)=\{D(x,y),D(x,z),D(y,z)\}$, for all
$x,y,z\in X$ pairwise distinct. 
Moreover, since $D_Y$ is a symbolic ultrametric on any subset $Y \subseteq X$ of size five, 
and given that the property of being a symbolic ultrametric is based on a 4-point condition, 
we have that such a map $D$, if it exists, is also a symbolic ultrametric. Thus, if $D$ 
exists, then $\delta$ is a three-way symbolic
 ultrametric.

To show that $D$ exists, assume for contradiction that there 
exist $x$ and $y$ in $X$ and two distinct subsets $Y_1$ and $Y_2$ of $X$ 
of size five, both containing $x$ and $y$, such that $D_{Y_1}(x,y) \neq D_{Y_2}(x,y)$. 
We may assume without loss of generality that $I=Y_1 \cap Y_2$ has size four. 
Moreover, we claim that $x$, $y$, $Y_1$ and $Y_2$ can be chosen in such a way that $\delta|_I$ 
is not of type $\hat{\delta_3}$, as defined in Table~\ref{tsub4}.

To prove this claim, consider the case where 
$\delta|_I$ is of type $\hat{\delta}_3$ 
(otherwise, the claim trivially holds). 
Assume $Y_1=\{x,y,z,t,u_1\}$ and $Y_2=\{x,y,z,t,u_2\}$, which 
implies $I=\{x,y,z,t\}$. Both the subtree of 
$\mathcal T_{Y_1}$ 
induced by $I$ and the subtree of $\mathcal T_{Y_2}$ induced by $I$ are of the 
form $\mathcal T_3$ in Figure~\ref{fsub4}, and their underlying phylogenetic trees are not isomorphic. We 
can assume that one has cherries  $\{x,y\}$ and $\{t,z\}$ and the 
other has cherries  $\{x,z\}$ and $\{t,y\}$. 
Then, we have not only that  $D_{Y_1}(x,y) \neq D_{Y_2}(x,y)$, but also that  
$D_{Y_1}(x,z) \neq D_{Y_2}(x,z)$,  $D_{Y_1}(z,t) \neq D_{Y_2}(z,t)$, and  
$D_{Y_1}(y,t) \neq D_{Y_2}(y,t)$. Moreover, it is easy to check that there 
exists a subset $Y^* \subset I$ of size three such that neither 
$\delta|_{Y^* \cup \{u_1\}}$ nor 
$\delta|_{Y^* \cup \{u_2\}}$ 
is of type $\hat{\delta_3}$.

Since $Y^*$ is a subset of $I$ of size three and, 
in view of the 
four inequalities listed above, there exists two elements $x',y' \in Y^*$ 
such that $D_{Y_1}(x',y') \neq D_{Y_2}(x',y')$. If we denote by $Y'$ the set 
$Y^* \cup \{u_1\} \cup \{u_2\}$, we have that both $Y' \cap Y_1$ 
and $Y' \cap Y_2$ have size four, and that at least one of 
$D_{Y'}(x',y') \neq D_{Y_1}(x',y')$ or $D_{Y'}(x',y') \neq D_{Y_2}(x',y')$ holds.
If the first inequality holds, the claim is then satisfied for $x',y', Y'$ and $Y_1$. 
Otherwise, it is satisfied for $x',y', Y'$ and $Y_1$, which completes the proof
of the claim.

Now, in light of the claim, 
the representation $\mathcal T_I$ of $\delta|_I$ is unique, and so 
is the symbolic ultrametric $D^I$ that is represented by $\mathcal T_I$. Moreover, $D^I$ is precisely 
the restriction of $D_{Y_1}$ to $I$, and the restriction of $D_{Y_2}$ to $I$. 
In particular, we have $D(x,y)=D_{Y_1}(x,y)$ and $D(x,y)=D_{Y_2}(x,y)$, which
contradicts $D_{Y_1}(x,y) \neq D_{Y_2}(x,y)$. 
\end{proof}

\section{A five-point characterization of three-way symbolic ultrametrics}\label{sec:five}

We now focus on using the results in the previous two
 sections to derive conditions
for characterizing  three-way symbolic ultrametrics that are 
analogous to conditions (U1) and (U2) for symbolic ultrametrics.

In the following, we shall consider expressions of the form $\sum_{m\in M} \alpha_m m$,
where $\alpha_m$ is a real number,  which arise 
when we take linear combinations of 
multisets in $\mathcal M$. We shall say that such an expression $\sum_{m\in M} \alpha_m m$  is 
\emph{valid for $M$} if the coefficient for each element in $M$ is 
contained in $\mathbb N$. For example, for $M=\{a,b\}$, 
if $S_1=2a+b$, $S_2=2b+a$ and $S_3=3a$ are multisets in $\mathcal M$, then 
we have $\frac{1}{3}(S_1+S_2)=a+b$, which  
is valid for $M$, but $S_3-S_1=a-b$ and $\frac{1}{2}(S_1+S_3)=\frac{5}{2}a+\frac{1}{2}b$
which are not valid for $M$. 

Now, suppose that $\delta: {X \choose 3} \to \mathcal M$
is a three-way symbolic map where $|X|\geq 5$.
Let $Y=\{x,y,z,u,v\}$ be a subset of $X$. Let $\nu_Y(\delta)$ 
denote the vector
\[(\delta(x,y,z),\delta(x,y,u),\ldots,\delta(z,u,v)).\]
In addition, suppose that $D_Y : {Y \choose 2} \to M$ is a map such that 
$$
\delta(a,b,c)=\{D_Y(a,b),D_Y(a,c),D_Y(b,c)\}
$$ 
for all $a,b,c \in Y$, and let $\mu_Y(\delta)$ denote the vector
\[(D_Y(x,y), D_Y(x,z), \ldots, D_Y(u,v)).\]

\noindent By definition of $D_Y$, it is 
straight-forward to check that $A \mu_Y(\delta)=\nu_Y(\delta)$, where

\[A=\begin{pmatrix}
1 & 1 & 0 & 0 & 1 & 0 & 0 & 0 & 0 & 0 \\
1 & 0 & 1 & 0 & 0 & 1 & 0 & 0 & 0 & 0 \\
1 & 0 & 0 & 1 & 0 & 0 & 1 & 0 & 0 & 0 \\
0 & 1 & 1 & 0 & 0 & 0 & 0 & 1 & 0 & 0 \\
0 & 1 & 0 & 1 & 0 & 0 & 0 & 0 & 1 & 0 \\
0 & 0 & 1 & 1 & 0 & 0 & 0 & 0 & 0 & 1 \\
0 & 0 & 0 & 0 & 1 & 1 & 0 & 1 & 0 & 0 \\
0 & 0 & 0 & 0 & 1 & 0 & 1 & 0 & 1 & 0 \\
0 & 0 & 0 & 0 & 0 & 1 & 1 & 0 & 0 & 1 \\
0 & 0 & 0 & 0 & 0 & 0 & 0 & 1 & 1 & 1 \\
\end{pmatrix}.\]

\noindent Note that in \cite{HHMS12} it was shown that
the matrix $A$ is invertible with inverse

\[A^{-1}=\frac{1}{6}\begin{pmatrix}
2 & 2 & 2 & -1 & -1 & -1 & -1 & -1 & -1 & 2 \\
2 & -1 & -1 & 2 & 2 & -1 & -1 & -1 & 2 & -1 \\
-1 & 2 & -1 & 2 & -1 & 2 & -1 & 2 & -1 & -1 \\
-1 & -1 & 2 & -1 & 2 & 2 & 2 & -1 & -1 & -1 \\
2 & -1 & -1 & -1 & -1 & 2 & 2 & 2 & -1 & -1 \\
-1 & 2 & -1 & -1 & 2 & -1 & 2 & -1 & 2 & -1 \\
-1 & -1 & 2 & 2 & -1 & -1 & -1 & 2 & 2 & -1 \\
-1 & -1 & 2 & 2 & -1 & -1 & 2 & -1 & -1 & 2 \\
-1 & 2 & -1 & -1 & 2 & -1 & -1 & 2 & -1 & 2 \\
2 & -1 & -1 & -1 & -1 & 2 & -1 & -1 & 2 & 2 \\
\end{pmatrix}.\]

Consider the product $\mu_Y(\delta)=A^{-1}\nu_Y(\delta)$. Then,
as the rows of $A^{-1}$ are indexed by pairs of distinct elements in $Y$,
it is straight-forward to check by considering 
the $\{p,q\}$th row of $A^{-1}$ (for $p \neq q \in Y$)
and
putting $\{e,f,g\} =Y-\{p,q\}$ and 
\begin{small}
\[
S_{p,q}^Y(\delta) = \frac{1}{6}(2(\delta(p,q,e)+\delta(p,q,f)+\delta(p,q,g)+\delta(e,f,g))-
\displaystyle\sum_{a,b \in Y-\{p,q\}}(\delta(p,a,b)+\delta(q,a,b))),
\]
\end{small}

\noindent that $ S_{p,q}^Y = \{ D_Y(p,q) \}$.
Defining $S_{p,q}^Y$ as above for $Y \subseteq X$ 
with $|Y|=5$ and $p\neq q \in Y$ we also have:

\begin{proposition}\label{prtodist}
Suppose that $|X|\geq 5$, that
$\delta : {X \choose 3} \to \mathcal M$
is a three-way symbolic map, and that $Y\subseteq X$ has size five. 
There exists a map $D^Y : {Y \choose 2} \to M$ such that 
$\delta(a,b,c)=\{D^Y(a,b),D^Y(a,c),D^Y(b,c)\}$ for all $a,b,c \in Y$ if and only 
if for all $p,q \in Y$ distinct, $S_{p,q}^Y(\delta)$ 
is valid for $M$, in which case $S_{p,q}^Y(\delta)$ is a singleton multiset.
\end{proposition}

\begin{proof}[Proof]
Suppose first that the map $D^Y$ exists. 
Without loss of generality we may assume that $D^Y=D_Y$.
In view of
the discussion preceding the proposition, 
it follows that $S_{p,q}^Y(\delta)$ is valid for $M$ 
for all $p,q \in Y$ distinct, as 
$S^Y_{p,q}(\delta) = \{D_Y(p,q)\}$.

To see the converse, assume that $S_{p,q}^Y(\delta)$ 
is valid for $M$ for all $p \neq q\in Y$. Fix $p$ and $q$. 
We claim that $S^Y_{p,q}(\delta)$ is a singleton multiset. 
To see this, put 
$\mathcal A= 2(\delta(p,q,e)+\delta(p,q,f)+\delta(p,q,g)+\delta(e,f,g))$
and 
$\mathcal B=\sum_{a,b \in Y-\{p,q\}}(\delta(p,a,b)+\delta(q,a,b))$. 
Then since $S^Y_{p,q}(\delta)$ is valid for $M$, every element
in $\mathcal B$ must also be an element in $\mathcal A$.
Hence, $S^Y_{p,q}(\delta)$ must contain 
$\frac{1}{6}|\mathcal A-\mathcal B| =1$ element
as $|\mathcal A|=24$ and $|\mathcal B|=18$. 
This proves the claim. 

Now, it is straight forward to see that if $S^Y_{p,q}(\delta)= \{s^Y_{p,q}\}$, for $p \neq q \in Y$,
then the  map $D^Y:{Y\choose 2}\to M$  defined  by putting $D^Y(p,q)=s^Y_{p,q}(\delta)$,
for all $p \neq q\in Y $,  satisfies the stated  property.
%
%
\end{proof}

We now present conditions for characterizing when 
a three-way symbolic map is a three-way symbolic ultrametric. 
For $\Sigma \in \mathcal M$,
we define the 
elements $m(\Sigma)$ and $n(\Sigma)$ of $M$ as follows:
\begin{itemize}
\item If $\Sigma$ contains a single element $A \in M$ repeated three 
times, we put  $m(\Sigma)=n(\Sigma)=A$.
\item If $\Sigma$ contains two distinct elements, we define $m(\Sigma)$ 
as the element of $\Sigma$ appearing twice and $n(\Sigma)$ as the element appearing only once.
\item If $\Sigma$ contains three distinct elements, we put $m(\Sigma)=n(\Sigma)=\emptyset$.
\end{itemize}
Note that if $\Sigma$ contains two or fewer distinct elements, 
then $\Sigma=\{m(\Sigma),m(\Sigma),n(\Sigma)\}$.

\begin{theorem}\label{theo:P1-P3}
Suppose that $|X|\geq 5$ and that
$\delta : {X \choose 3} \to \mathcal M$
is a three-way 
symbolic map. Then $\delta$ is a three-way symbolic ultrametric 
if and only if the following hold:
\begin{enumerate}
\item[(P1)] For all subsets $Y\subseteq X$ of size
five and all $x,y\in Y$ distinct, $S_{x,y}^Y(\delta)$ is valid for $M$.
\item[(P2)] For all pairwise distinct $x,y,z \in X$, $\delta(x,y,z)$ contains at most two distinct elements.
\item[(P3)] For all pairwise distinct $x,y,z,u \in X$ with $\delta(x,y,z)=\delta(y,z,u) \neq \delta(x,y,u)=\delta(x,z,u)$ holding, we have $m(\delta(x,y,z))=m(\delta(x,y,u))$.
\end{enumerate}
\end{theorem}

\begin{proof}[Proof]
Assume first that $\delta$ is a three-way symbolic 
ultrametric. By Theorem~\ref{5enough} and Proposition~\ref{prtodist} it follows that 
Properties  (P1) and (P2) must hold. 
To see that Property (P3) holds too let $\{x,y,z,u \}\in 
{X\choose 4} $ 
be such that $\delta(x,y,z)=\delta(y,z,u) \neq
 \delta(x,y,u)=\delta(x,z,u)$. 
Since $\delta|_{\{x,y,z,u\}}$ is a three-way symbolic ultrametric,
Proposition~\ref{rep4} combined with Table~\ref{tsub4}
implies that 
$\delta|_{\{x,y,z,u\}}$ is either of type
$\hat{\delta}_3$ and $\hat{\delta}_5$. 
Clearly, $m(\hat{\delta}_i(x,y,z))=
m(\hat{\delta}_i(x,y,u))$
holds for $i=3,5$ and, so, Property (P3) follows.

Conversely, assume that $\delta$ satisfies Properties 
(P1) -- (P3). 
Consider a subset $Y \subseteq X$ of size five. By Proposition~\ref{prtodist}, there exists a map 
$D^Y : {Y \choose 2} \to M$ such that $\delta(x,y,z)=\{D^Y(x,y),D^Y(x,z),D^Y(y,z)\}$ for all $x,y,z \in Y$. 
We claim that $D^Y$ is a symbolic ultrametric.
For this it suffices to show that $D^Y$ satisfies 
Property~(U2) as 
Property ~(U1) is a direct consequence of Property~(P1).

To see that $D^Y$ satisfies Property~(U2), assume for
contradiction that there exist pairwise distinct
$x,y,z,u \in Y$ such
that $D^Y(x,y)=D^Y(y,z)=D^Y(z,u) \neq D^Y(z,x)=D^Y(x,u)=D^Y(u,y)$. 
Put $A=D^Y(x,y)$ and $B=D^Y(z,x)$. Then 
$\delta(x,y,z)=\delta(y,z,u)=2A+B \neq A+2B= \delta(x,y,u)=\delta(x,z,u)$. Since, 
$m(\delta(x,y,z))=A\neq B= m(\delta(x,y,u))$ also holds 
 this is impossible in view of Property~(P3). Thus, 
 $D^Y$ also satisfies Property~(U2) and, so, is a 
 symbolic ultrametric, as claimed.

Since $D^Y$ is a symbolic ultrametric, there exists 
a labelled rooted tree $\mathcal T$ that represents 
$D^Y$. Combined with the definition of $D^Y$
it follows that $\mathcal T$ also represents $\delta|_Y$.
Thus, $\delta|_Y$ is a three-way symbolic ultrametric 
and, so, $\delta|_Y$ is a three-way symbolic ultrametric
for all subsets $Y \subseteq X$ with $|Y|=5$.
 By Theorem~\ref{5enough}, it follows that 
$\delta$ is a three-way symbolic ultrametric.
\end{proof}

Note that Properties~(P1) -- (P3) are independent 
of each other. Indeed, that Property~(P2) is independent 
of Properties~(P1) and (P3) and that Property~(P3) is independent of Properties~(P1)
and (P2) is a direct consequence of the fact that
Properties~(U1) and (U2) are independent of each other.

To see that Property~(P1) is independent of 
Properties~(P2) 
and (P3), consider the  three-way symbolic map 
$\delta : {X \choose 3} \to \mathcal M_{\{A,B\}}$
defined, for all $x,y,z \in X$, by putting
$\delta(x,y,z)=2A+B$. The 
map $\delta$ always satisfies (P2) and (P3), but 
if $|X| \geq 5$, $\delta$ does not satisfy (P1).

\section{Reconstructing three-way symbolic ultrametric 
representations using triplets}\label{sec:triples}

In this section we are interested in 
determining when a three-way symbolic map $\delta$ on $X$
is a tree-map or  a symbolic ultrametric.
Clearly, using the conditions given in Theorem~\ref{theo:P1-P3}
this can be done by examining every subset of $X$ with size five. 
However, we now show how to do this using a triplet-based approach, 
which essentially reduces the problem to considering subsets of $X$ 
of size three. 

Recall that a {\em triplet} is a binary phylogenetic tree on three leaves. By $xy|z$ we denote
the triplet with leaf-set $\{x,y,z\}$ which has $x,y$ adjacent to the same vertex in the tree.
For $T$ a phylogenetic tree on $X$ and $x,y,z \in X$, we say that $T$ \emph{displays} 
the triplet $xy|z$ if 
$\mathrm{lca}_T(x,z)=\mathrm{lca}_T(y,z) \neq \mathrm{lca}_T(x,y)$. To keep notation at bay,
we sometimes also say that a labelled
tree $\mathcal T=(T,t)$ on $X$ {\em displays} a triplet $r$
if $r$ is displayed by $T$.

In \cite[Section 7.6]{SS03} a triplet-based approach is described
for deciding whether or not a two-way symbolic map $\delta$ is a symbolic 
ultrametric or not and, if it is, for building a 
labelled tree which represents $\delta$.
This approach is based on the BUILD algorithm, that was presented under that name in \cite[p.407]{A81}. 
Using the results in Section~\ref{sec:tree}, the
BUILD algorithm also allows us to 
check if a three-way symbolic map
is a tree-map using triplets as follows. 
Suppose $\delta:{X\choose 3} \to \mathcal M$ 
is a three-way symbolic map.
Pick any $r \in X$. Then, using the BUILD-based approach, 
we can check whether or not the map $\delta_r$  defined 
in Section~\ref{sec:tree} is
a symbolic ultrametric by taking the set 
of triplets $xy|z$ with $x,y,z\in X$ distinct, for 
which 
$\delta_r(x,y) \neq \delta_r(x,z) = \delta_r(y,z)$ holds
as input to BUILD. 
If this is not the case, then by Lemma~\ref{lmddr}, 
$\delta$ is not a three-way symbolic tree-map. Otherwise, if $(T,t)$
is the representation of $\delta_r$ returned by BUILD, 
then we can simply check whether or
not this leads to a representation of $\delta$ 
by attaching the
leaf $r$ to the root of $T$. If this is possible then 
$\delta$ is a three-way symbolic tree-map, otherwise 
it is not. 

We now turn our attention to three-way symbolic ultrametrics. We begin by presenting a key
link between triplets and such maps whose proof is straight
forward. Denote the underlying set of a multiset $\mathcal A$ 
by $\underline{\mathcal A}$.

\begin{proposition}\label{prop:prtrip}
	Let $\mathcal T=(T,t)$ be a discriminating labelled tree. 
	For $x,y,z \in X$ distinct:
	\begin{enumerate}
		\item[(T1)] If $xy|z$ is a triplet displayed by $T$, then $t(\mathrm{lca}_T(x,z))=t(\mathrm{lca}_T(y,z))=
		m(\delta_{\mathcal T}(x,y,z))$ and $t(\mathrm{lca}_T(x,y))=n(\delta_{\mathcal T}(x,y,z))$
		\item[(T2)] If $T$ does not display any triplet on $\{x,y,z\}$, then $|\underline{\delta_{\mathcal T}(x,y,z)}|=1$.
	\end{enumerate}
\end{proposition}

\begin{corollary}\label{correcov}
	Up to isomorphism, any labelled tree $\mathcal T$
	can be uniquely reconstructed from $\delta_{\mathcal T}$
	and the set of triplets displayed by $\mathcal T$.
\end{corollary}

\begin{proof}[Proof]
	Put $\mathcal T=(T,t)$ and $\delta=\delta_{\mathcal T}$. Let $t:V^o(T)\to M$
	and let $\mathcal R$ denote the set of triplets displayed by $\mathcal T$.
	Define a map $D_{\delta}:{X\choose 2}\to M$ as follows.
	Suppose $x,y\in X$ distinct.
	If there exists some $z\in X-\{x,y\}$ 
	such that no triplet on $\{x,y,z\}$ is
	contained in $\mathcal R$, then define $D_{\delta}(x,y)$ 
	to be the element in $\delta_{\mathcal T}(x,y,z)$. If
	there exists some $z\in X-\{x,y\}$ 
	such that $xy|z\in \mathcal R$
	then put $D_{\delta}(x,y)=n(\delta_{\mathcal T}(x,y,z))$
	and if $xz|y\in\mathcal R$ then
	put $D_{\delta}(x,y)=m(\delta_{\mathcal T}(x,y,z))$.
	In view of Proposition~\ref{prop:prtrip}, 
	the map $D_{\delta}$ is clearly well-defined. 
	
	The corollary now follows in view of Theorem~\ref{thBD} as $D_{\delta}$ 
	is equal to the symbolic ultrametric $D_{\mathcal T}$ 
	that is represented by $\mathcal T$ (as 
	$D_{\delta}(x,y)=t(lca(x,y))=D_{\mathcal T}(x,y)$ clearly holds 
	for all $x,y\in X$ distinct).
%
\end{proof}

In light of Corollary~\ref{correcov}, it
is of interest to understand when, for a labelled tree $\mathcal T$,
the set of triplets displayed by $\mathcal T$ can be obtained from $\delta_{\mathcal T}$. 
The tree $\mathcal T_3$ in Figure~\ref{fsub4}, 
suggests that this is not always possible. In fact, as we shall 
show, it suffices to exclude a special type of 
labelled tree which we define next. 

A {\em fixed-cherry tree on $X$ (with cherry $\{x_1,x_2\}$)}, $|X|\geq 4$,  is a labelled 
tree $\mathcal T=(T,t)$ on $X$ 
such that the root $\rho_T$ of $T$ has 
two children $v$ and $w$
with $t(v)=t(w)\neq t(\rho_T)$, $v$ is the 
parent of two elements $x_1$ and $x_2$ of $X$,
and $w$ the parent of all elements in $X-\{x_1,x_2\}$. 
For example, the  tree $\mathcal T_3$ in 
Figure~\ref{fsub4} is a fixed-cherry tree on 
$X=\{1,2,3,4\}$ with cherry $\{1,2\}$. 
Note that if $\mathcal T=(T,t)$  is a fixed-cherry tree with cherry $\{x_1,x_2\}$ and 
$x,y,z\in X$ distinct, then $\delta_{\mathcal T}(x,y,z)=\{t(w),t(w),t(w)\}$ 
if neither $x_1$ nor $x_2$ belong to $\{x,y,z\}$ 
and $\delta(x,y,z)=\{t(\rho_T),t(\rho_T), t(w)\}$ else. 
We call a  three-way symbolic map $\delta:{X\choose 3}\to \mathcal M$ 
that satisfies these conditions for some  $x_1 \neq x_2 \in X$ 
a {\em fixed cherry map (with cherry $\{x_1,x_2\}$)}. 
The following observation is straight-forward to check.

\begin{lemma}\label{oct}
	Suppose that $|X|\geq 5$ and that
	$\delta$ is a three-way symbolic map on $X$. Then $\delta$ can be represented by a 
	fixed-cherry tree on $X$
	with cherry $\{x_1,x_2\}$  if and only if $\delta$ is a fixed-cherry 
	map with cherry $\{x_1,x_2\}$ .
\end{lemma}

Note that a triplet $xy|z$ with  
$x,y,z \in X$ is displayed by a fixed-cherry tree on $X$  with 
cherry $\{x_1,x_2\}$ if and 
only if either $\{x,y\}=\{x_1,x_2\}$ or $z \in \{x_1,x_2\}$,
and $ x,y \in X-\{x_1,x_2\}$ hold. 
In particular, if $|X| > 4$ and $\delta$ is a fixed-cherry map,
then the cherry can be easily identified from $\delta$, and 
therefore also all of the triplets displayed by $\mathcal T$.

We now consider how to obtain the triplets displayed
by a labelled tree
$\mathcal T$ in case $\mathcal T$ is not a fixed-cherry tree.
We start with a useful lemma.
Suppose $\mathcal T = (T,t)$ is a labelled tree and $Y \subseteq X$, $|Y|\ge 4$, is such that 
$\delta_{\mathcal T}|_Y$ has a
unique discriminating representation. 
Then we denote that representation by
 $\mathcal T_Y=(T_Y, t_Y)$. 
 
\begin{lemma}\label{lmsubset}
	Let $\mathcal T$ be a discriminating labelled tree
	 on $X$ and assume that $Y\subseteq X$ is such that
	 $\delta_{\mathcal T}|_Y$ has a unique discriminating 
	 representation. 
	If $t$ is a triplet displayed by $\mathcal T_Y$, 
	then $t$ is displayed by $\mathcal T$. 
\end{lemma}

\begin{proof}[Proof]
	Put $\delta=\delta_{\mathcal T}$ and $\mathcal T=(T,t)$.
	It suffices to note that $\mathcal T_Y$ is obtained from $\mathcal T$ 
	by first taking the subtree $T'$ of $T$ induced by $Y$, and then 
	collapsing edges of $T'$ both of whose end vertices 
	have the same label under the 
	restriction $t'$ of $t$ to $V(T')$. Clearly,
	$\mathcal T_Y$  is a discriminating representation of $\delta|_Y$. By assumption, it follows that $\mathcal T_Y$
	is the unique discriminating 
	 representation of $\delta|_Y$.
	
	It is well-known \cite[Theorem 6.4.1]{SS03} that the set 
	$\mathcal R$ of	triplets displayed by $T'$ is contained in the set
	of triplets displayed by $T$. Since 
	the process of collapsing edges of $T'$ 
	removes triplets from $\mathcal R$, 
	but does not add any, it follows that 
	a triplet displayed by $T_Y$ is also displayed by $T$.
\end{proof}

We now present the main result of this section.

\begin{theorem}\label{thtrip}
	Suppose that $|X|\geq 4$ and that 
	$\mathcal T$ is a labelled tree on $X$
	that is not a fixed-cherry tree.
	Then, for all $x,y,z \in X$ distinct, $\mathcal T$ 
	displays the triplet $xy|z$ if and only 
	if one of the following two properties holds:
	\begin{enumerate}
		\item[(P1)] There exists some $u \in X$ such that 
		$\delta_{\mathcal T}(x,u,z)=\delta_{\mathcal T}(y,u,z) \neq \delta_{\mathcal T}(x,y,u)$ and if $|\underline{\delta_{\mathcal T}(x,y,u)}|=1$ then 
		$ \delta_{\mathcal T}(x,y,u) \neq 
		\delta_{\mathcal T}(x,y,z)$.
		\item[(P2)] There exists some $u \in X$ such that $|\{\delta_{\mathcal T}(x,u,z),\delta_{\mathcal T}(y,u,z),\delta_{\mathcal T}(x,y,u)\}|=3$ and $m(\delta_{\mathcal T}(x,u,z))=m(\delta_{\mathcal T}(y,u,z)) \neq m(\delta_{\mathcal T}(x,y,u))$.
	\end{enumerate}
\end{theorem}

\begin{proof}[Proof]
	Put $\mathcal T=(T,t)$ and $\delta=\delta_{\mathcal T}$.
	Assume first that $x,y,z \in X$ distinct
	are such that $T$ displays 
	the triplet $xy|z$. Put $v=\mathrm{lca}(x,z)$ 
	and $w=\mathrm{lca}(x,y)$. We proceed using a case-analysis on the 
	structure of $T$. Since $\mathcal T$ is not a fixed-cherry tree we need to consider the following (not necessarily disjoint) cases:
	(a): $w$ is not a child of $v$, (b): $v$ is not the root of $T$ or has outdegree three or 
	more, (c): $w$ has a child that is neither $x$ nor $y$, and (d): 
	there exists a vertex $v_0$ on the path from $v$ to $z$ with $t(v_0) \neq t(w)$.
	
	\begin{figure}[h]
		\begin{center}
			\includegraphics[scale=0.7]{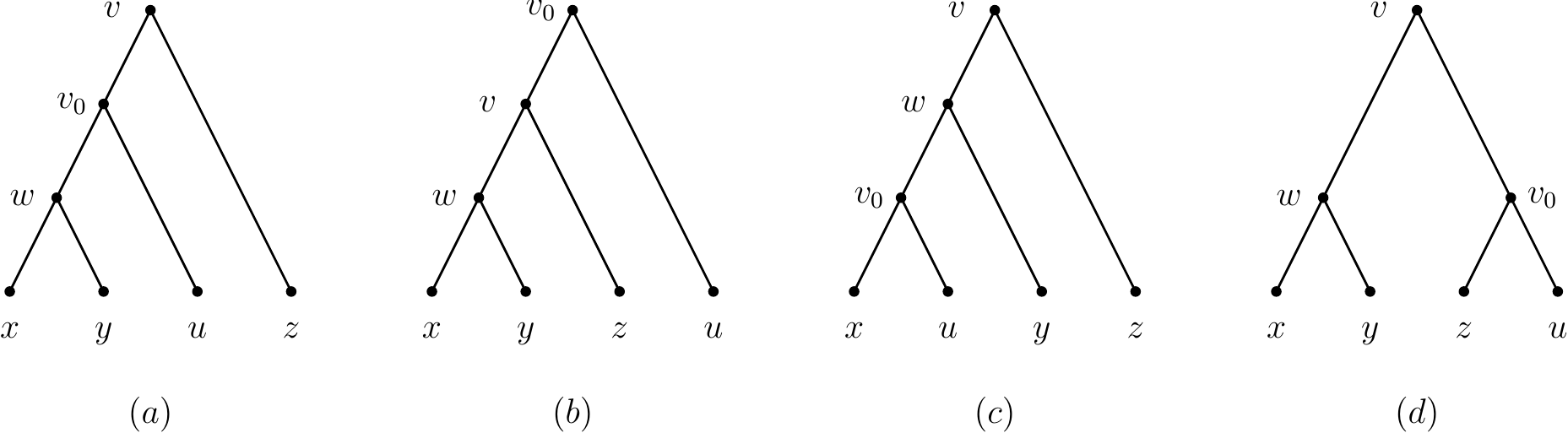}
			\caption{Cases (a)-(d) for the case-analysis carried out 
				in the proof of Theorem~\ref{thtrip}. See text for details.}
			\label{prtrip}
		\end{center}
	\end{figure}
	
	Case (a): Consider the parent $v_0$ of $w$, and an element $u$ in $X$ 
	that is below $v_0$ but not below $w$ (see
	 Figure~\ref{prtrip}(a)). 
	Since $\mathcal T$ is a discriminating representation 
	for $\delta_{\mathcal T}$,
	 we have $t(v_0) \neq t(w)$. Hence,
	$\delta(x,u,z)=\delta(y,u,z)=\{t(v_0), t(v), t(v)\}$ and $\delta(x,y,u)=\{t(w),t(v_0), t(v_0)\}$.
	Consequently, $\delta(x,u,z)=\delta(y,u,z)\not=
	\delta(x,y,u)$. Note that if $t(v)=t(w)$, then 
	$\delta(x,y,z)=\{t(w),t(v),t(v)\}$ and, so, 
$|\underline{\delta(x,y,z)}|=1$. But then 
$\delta(x,y,u)\neq \delta(x,y,z)$
as $|\underline{\delta(x,y,u)}|=2$. Hence, the second 
condition in 
Property~(P1) holds, too. So assume that $t(w)\not=t(v)$.
Then $|\underline{\delta(x,y,u)}|=2$ and so the second 
condition in Property~(P1) does not apply.
	%
	
	Case (b): Consider an element of $u\in X$ such that $v_0:=\mathrm{lca}(u,z)=lca(u,x)$ 
	(see Figure~\ref{prtrip}(b)). If $w$ is not a child of $v$ then
	Property~(P1) follows by Case~(a). So assume that $w$ 
	is a child of $v$. Then $t(v)\neq t(w)$ as $\mathcal T$ is a 
	discriminating representation for $\delta_{\mathcal T}$.
	Since $\delta(x,u,z)=\delta(y,u,z)=\{t(v), t(v_0), t(v_0)\}$ 
	and $\delta(x,y,u)=\{t(w),t(v_0), t(v_0)\}$ we have $\delta(x,u,z)=\delta(y,u,z)\neq \delta(x,y,u)$. Since the choice of $v_0$ implies that 
	$|\underline{\delta(x,y,u)}|\neq 1$ the
	second condition in Property~(P1) does not apply. Hence,
	Property~(P1) is also satisfied in this case.
	
	Case (c): Then there is some $u \in X$ below $w$ that is neither $x$ nor $y$. We 
	may assume without loss of generality that $w=\mathrm{lca}(y,u)$. Put
	$v_0=\mathrm{lca}(x,u)$ (see Figure~\ref{prtrip}(c)). 
	Note that $v_0=w$ 
	may hold. Clearly, $\delta(x,u,z)=\{t(v_0),t(v),t(v)\}$, $\delta(y,u,z)=\{t(w),t(v),t(v)\}$ 
	and $\delta(x,y,u)=\{t(v_0),t(w),t(w)\}$. If $v_0\neq w$ 
	then $\delta(y,u,z)\neq \delta(x,u,z)\not=\delta(x,u,y)$.
	Hence, $|\{\delta(y,u,z), \delta(x,u,z),\delta(x,u,y) \}|=3$.
	Since $m(\delta(x,u,z))=t(v)=m(\delta(y,u,z))$
	and $m(\delta(x,y,u))=t(w)$, Property~(P2) follows.
	So assume that $v_0\not=w$. Then 
	$\delta(y,u,z)=\delta(x,u,z)=\{t(w),t(v),t(v)\}$
	and $\delta(y,u,x)=\{t(w),t(w),t(w)\}$. In view of
	Property~(P1) holding if Case (a) applies, we may assume
	without loss of generality that $w$  is a child of $v$. 
	Since $\mathcal T$ is a discriminating representation
	of $\delta_{\mathcal T}$ we have 
	$t(v)\not=t(w)$. Hence,
	$\delta(y,u,z)=\delta(x,u,z)\neq \delta(x,y,u)$.
	Since $|\underline{\delta(x,y,z)}|=1$ and $|\underline{\delta(x,y,z)}|\neq 1$,
	Property~(P1) follows in this case, too.
	
	Case (d):  Let $u \in X$ such that $v_0=\mathrm{lca}(z,u)$ 
	(see Figure~\ref{prtrip}(d)). Then  $\delta(x,u,z)=\delta(y,u,z)=\{t(v_0),t(v),t(v)\}$ 
	and $\delta(x,y,u)=\{t(w),t(v),t(v)\}$. If 
	$t(w)= t(v_0)$ we have 
	$\delta(x,u,z)=\delta(y,u,z)=\delta(x,y,u)$. In view of 
	Property~(P1) holding if Case~(a) applies, we may 
	assume without loss of generality that $w$ is a child of $v$. Hence, $t(w)\not= t(v)$ because $\mathcal T$ 
	is a discriminating representation for $\delta$.
	But then $|\underline{\delta(x,y,u)}|\neq 1$ and, so, 
	the second condition in Property~(P1) does not apply.
	%

	Conversely, let $x,y,z\in X$ distinct. Assume first 
	that there exists some $u\in X-\{x,y,z\}$ such that 
	Property~(P1) is satisfied for the namesakes of $u$, $x$,
	$y$, and $z$. 
	Consider the restriction $\delta'$ of $\delta$ to $\{x,y,u,z\}$. Let $\mathcal T'=(T',t')$ denote 
	a discriminating representation of $\delta'$.
	To see that $xy|z$ is displayed by $T$ we
	claim first that $\mathcal T'$ is the unique 
	discriminating representation of $\delta'$.
	To see the claim, we show that $xy|z$ is displayed by $T'$.
	Assume for contradiction that the triplet $xy|z$ is not displayed by $T'$.
	In view of the first condition in Property~(P1),
	the outdegree of the root $\rho_{T'}$ cannot be four.
	Hence, one of the triplets $x|yz$ and $y|xz$ must be
	displayed by $T'$ and $T'$ is either resolved or unresolved.
	Assume first that $T'$ is resolved. Then a straight
	forward case analysis concerned with adding $u$ to
	the triplet $x|yz$ implies that that triplet cannot be 
	displayed by $T'$. Swapping the roles of $x$ and 
	$y$ in that argument also implies that 
	the triplet $y|xz$ cannot be displayed by $T'$ either. 
	Thus, $T'$ must be unresolved and, so, either $\rho_{T'}$
	has outdegree three or one of the children of $\rho_{T'}$
	has outdegree three. 
	
	If $T'$ displays the triplet
	$x|yz$ and the outdegree of $\rho_{T'}$ is three
	then $|\underline{\delta(y,u,z)}|=1$. Hence, 
	$\delta(x,y,u)=\delta(x,y,z)$ in view of 
	the second condition in Property~(P1)
	which is impossible. Thus, one of the children of $\rho_{T'}$
	has outdegree three. But this is impossible in view of
	the first condition in Property~(P1). Similar arguments
	imply that the triplet displayed by $T'$ cannot be
	$y|xz$ either which is impossible. Thus, $T'$ must 
	display the triplet $xy|z$.	
	Consequently, either $\rho_{T'}$ is the parent of $u$ and
	$z$ or $x$, $y$, and $u$ have the same parent. In either
	case it follows that $\delta'$ cannot be 
	of type $\hat{\delta_3}$. Thus, $\mathcal T'$
is the unique discriminating representation of $\delta'$,
as claimed. By Lemma~\ref{lmsubset},
it follows that $xy|z$ must be displayed by $\mathcal T$.

	 %
%

Assume next 
that there exists some $u\in X-\{x,y,z\}$ such that 
Property~(P2) is satisfied for the namesakes of $u$, $x$,
$y$, and $z$.  Consider again the restriction $\delta'$ of $\delta$ to $\{x,y,u,z\}$. Then $\delta'$ must have a 
 representation $\mathcal T'=(T',t')$. In view of the first condition of Property~(P2), $\mathcal T'$ must be
 discriminating. A straight forward case
 analysis implies that $\delta'$ cannot be of type $\hat{\delta_3}$. Thus, $\mathcal T'$  
 is the unique discriminating representation of $\delta'$. 
 
 In view of
	Table~\ref{tsub4}, there must exist at least two subsets $Y$ 
	and $Y'$ of $\{x,y,z,u\}$ of size three satisfying $\delta(Y)=\delta(Y')$. 
	Since $|\{\delta(x,u,z),\delta(y,u,z),\delta(x,y,u)\}|=3$, it follows that $\{x,y,z\}$ must be 
	one of these subsets.
	If $\delta(x,y,z)=\delta(x,y,u)$, then 
	$D_{\mathcal T}(x,u)=m(\delta(x,u,z))$ 
	and $D_{\mathcal T}(y,u)=m(\delta(y,u,z))$ must hold
	where $D_{\mathcal T}$ is the symbolic ultrametric 
	represented by $\mathcal T$. Indeed, 
	since $\delta(x,y,u)=\delta(x,y,z)$, one of the 
	following two cases must hold:
	($\alpha$) $D_{\mathcal T}(x,z)=D_T(x,u)$ and $D_{\mathcal T}(y,z)=D_{\mathcal T}(y,u)$ and 
	($\beta$) $D_{\mathcal T}(x,z)=D_T(y,u)$ and $D_T(y,z)=D_{\mathcal T}(x,u)$.
However	Case~($\beta$) implies
	$\delta(x,z,u)=\delta(y,z,u)$, 
	which is 
	impossible in view of the
	assumption that $\{\delta(x,y,u),\delta(y,z,u),\delta(x,z,u)\}$ has size three. Thus, Case ($\alpha$) must
	hold. But then $D_{\mathcal T}(x,u)=m(\delta(x,u,z))$ and $D_{\mathcal T}(y,u)=m(\delta(y,u,z))$, as required.
	
	Since, by assumption, we also have 
	$m(\delta(x,u,z))=m(\delta(y,u,z))$ we obtain $D_{\mathcal T}(x,u)=D_{\mathcal T}(y,u)$. 
	Since $D_{\mathcal T}(x,u)$ and $D_{\mathcal T}(y,u)$ are both elements in the
	multiset  $\delta(x,y,u)$, 
	we obtain $m(\delta(x,u,z))=D_{\mathcal T}(x,u)=m(\delta(x,y,u))$, which is 
	impossible in view of (P2). Thus, we either have $\delta(x,y,z)=\delta(x,u,z)$ 
	or $\delta(x,y,z)=\delta(y,u,z)$. Note that the roles of $x$ and $y$ are 
	interchangeable here, so we may assume without loss of generality that $\delta(x,y,z)=\delta(y,u,z)$.
	
	Using similar arguments as before, we have $D_{\mathcal T}(x,z)=m(\delta(x,u,z))$, $D_{\mathcal T}(x,y)=m(\delta(x,u,y))$, 
	and $D_{\mathcal T}(y,z)=m(\delta(y,u,z))$ in this case. By 
Property~(P2), it follows that $D_{\mathcal T}(x,z)=D_{\mathcal T}(y,z)\neq D_{\mathcal T}(x,y)$. Thus,  $xy|z$ is displayed by $T'$ and, by Lemma~\ref{lmsubset}, $xy|z$ is also displayed 
	by $T$.
\end{proof}

We now explain how, as a 
direct consequence of Lemma~\ref{oct} and Theorem~\ref{thtrip}, it is possible to decide
whether or not a three-way symbolic map $\delta$ on a set $X$ with $|X|\ge 5$ is a
three-way symbolic ultrametric by considering triplets and, if so, construct 
the labelled tree $\mathcal T$  which represents  $\delta$.

First, check if $\delta$ is a fixed-cherry map.  
If this is the case, then $\delta$ is a three-way symbolic ultrametric and $\mathcal T$ can be easily constructed. 
If not, then
compute the set $\mathrm{Tr(\delta)}$ of triplets of $X$ 
satisfying Properties (P1) or (P2), and use it as input to the 
BUILD algorithm. If
there is no tree that display all the triplets in $\mathrm{Tr(\delta)}$, then $\delta$ is not 
a three-way symbolic ultrametric. Otherwise, using 
the tree $T$ that is constructed from the BUILD algorithm and the map $\delta$, 
it is straight-forward to decide if there is a labelling map $t$ for $T$ such that 
$(T,t)$ represents $\delta$.  If this is the case, then $\delta$ is a three-way symbolic ultrametric
which has the computed labelled tree $(T,t)$ as a representation, otherwise it is not.

Note that BUILD may return a tree $T$ 
from $\mathrm{Tr}(\delta)$ even if 
the map $\delta$ is not a three-way symbolic ultrametric. 
For example, let $M=\{A,B\}$
and consider the map $\delta : {X \choose 3} \to \mathcal M$
where $X=\{1,2,3,4,5\}$, and $\delta(x,y,z)=3A$ if $\{x,y,z\}=\{3,4,5\}$, and $\delta(x,y,z)=2A+B$ 
otherwise. Although the map is clearly not representable by a labelled tree, we 
have $\mathrm{Tr}(\delta)=\{34|1,34|2,35|1,35|2,45|1,45|2\}$, and 
it is easy to check that there exists a phylogenetic 
tree on $X$ whose set of displayed triplets is
$Tr(\delta)$.

\section{Conclusion}\label{sec:discuss}

We conclude by presenting some possible future directions:

\begin{itemize}
	\item In \cite{chepio} some relationships are derived between three-way and two-way 
	dissimilarities in general. It would be interesting to see which of these 
	relationships might be extendable to symbolic three-way maps. 
	\item We define three-way tree-maps in terms of medians in leaf-labelled 
	trees. Can any of our results be extended
	to median networks \cite{median}? Also, can our results concerning three-way symbolic ultrametrics 
	be extended to rooted phylogenetic networks? (cf. e.g. \cite{HS})
	\item We can clearly consider generalizations of three-way symbolic maps 
	to $k$-way symbolic maps, $k\ge 2$ (see e.g \cite{chepio,D00,W10}), 
	and therefore generalize the 
	notion of a three-way symbolic ultrametric in the natural way. 
	If  $\delta$ is a $k$-way symbolic map and its restriction to every $k+2$ subset is a
	$k$-way symbolic ultrametric, then is $\delta$ a symbolic ultrametric?
	\item In \cite{Stadler}, an application of symbolic ultrametrics to 
	constructing genome-based phylogenies is presented. 
	It would be interesting to see if this application 
	could be extended to three-way maps.
	Note that results presented in \cite{L06} might be relevant in this context. 
	Also, it would be interesting to develop associated algorithms such as those in \cite{LM15}, for
	three-way symbolic maps.
\end{itemize}

\vspace{0.5cm}

\noindent{\bf Acknowledgement:} The authors thank The Biomathematics Research Centre, University
of Canterbury, New Zealand, where some of this work was undertaken. The authors would like to thank  
Vladimir Gurvich and Stefan Gr\"unewald for making us aware of their work, and also the anonymous 
reviewers for their helpful feedback.  \\

\end{document}